\documentclass{tac}

\usepackage{amsmath}

\usepackage{amssymb}

\title{A Tale of Two Paths}

\address{Unaffiliated\\
 Chicago, Illinois, USA}

\author{Johnathon Taylor}

\usepackage{tikz}
\usepackage{tikz-cd}

\newtheoremrm{construction}{Construction}
\newtheoremrm{conjecture}{Conjecture}

\begin{document}

\maketitle

\begin{abstract}
We provide a new presentation for a cylinder object inside of the category of co-globular $\infty$-groupoids. We prove that the cylinder construction that Lanari provides is isomorphic to the one we build in this paper. 
\end{abstract}

\copyrightyear{2026}


\eaddress{jt3theend17@gmail.com}

\tableofcontents
\section*{Introduction}
Grothendieck hypothesized in \emph{Pursuing Stacks} \cite{Gr1} that there was a model structure on a locally presentable category whose objects are called Grothendieck $\infty$-groupoids. Grothendieck posited that Grothendieck $\infty$-groupoids canonically have an underlying globular set and are presented by objects that take the appearance of disks. Moreover, Grothendieck hypothesized that there was an equivalence of homotopy categories between Grothendieck $n$-groupoids and homotopy $n$-types for all $n\geq 0$, including $n=\infty$, in a statement called \emph{Grothendieck's Homotopy Hypothesis}.

Maltsiniotis provided a rigorous definition of Grothendieck $\infty$-groupoids in \cite{Geo2}. Moreover, the combined work of Ara, Bourke, Henry, Lanari, and Maltsiniotis have proven many notable results. First and foremost, Henry proved that there is a semi-model structure on the category of Grothendieck $\infty$-groupoids provided that a certain lemma called the \emph{pushout lemma} is true \cite{Henry2016}. Moreover, Henry proved that if the semi-model structure exists, then Grothendieck's Homotopy Hypothesis is true. Lanari constructed the underlying globular set of the \emph{path object} of an $\infty$-groupoid and proved that its existence implied that Grothendieck's semi-model structure exists \cite{Lanari2018,Lanari2020}. Joint work of Henry and Lanari showed that the Homotopy Hypothesis is true for dimension 3 \cite{HenLan}.

We focus in on the notion of cylinder object introduced by Lanari. One of the fundamental tools that Lanari implements is the suspension functor that he introduces in Section 4 of \cite{Lanari2020}. It is technically defined and we seek a more friendly construction (friendly is definitely used relatively) of the cylinder object. Therefore, we re-imagine the construction of a cylinder object. Lanari's cylinder construction works by inductively propagating the shapes of a cylinder up a dimension and then shaping the result via a colimit given by the shape of a zigzag diagram. If we let $\alpha$ be a $1$-disk, the visualization of this result is the following.

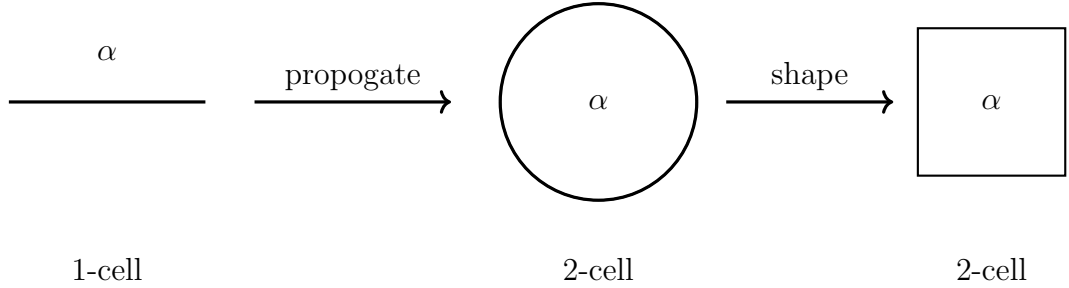
\[
\begin{tikzpicture}[scale=1.3]
\begin{scope}
\draw[-,very thick] (-5,0) -- (-3,0);
\node at (-4,0.5) {$\alpha$};
 \node at (-4,-1.7) {1-cell};
\end{scope}
\draw[->,very thick] (-2.5,0) -- (-0.5,0)
    node[midway,above] {propogate};

\begin{scope}
    \draw[thick] (1,0) circle (1);

    \draw[very thick] (0,0) arc (180:0:1);
    \draw[very thick] (0,0) arc (180:360:1);

    \node at (1,0) {$\alpha$};

    \node at (1,-1.7) {2-cell};
\end{scope}

\draw[->,very thick] (2.3,0) -- (4,0)
    node[midway,above] {shape};

\begin{scope}[xshift=5cm]
    \draw[thick] (-0.75,-0.75) rectangle (0.75,0.75);

    \node at (0,0) {$\alpha$};

    \node at (0,-1.7) {2-cell};
\end{scope}

\end{tikzpicture}
\]
Our construction that we build here is built with a fundamentally different process.

Our construction is done in two layers: the category of Grothendieck infinity groupoids and the category of co-globular $\infty$-groupoid objects. We begin by setting $C^{(0)}=D_\bullet \coprod D_\bullet$ inside of the category of co-globular $\infty$-groupoids. Then we inductively construct a chain inside of the category of co-globular $\infty$-groupoids whose colimit we call the \emph{cylinder object} and denote by $\mathbf{Cyl}_\bullet$.
\[
C^{(0)}
\to
C^{(1)}
\to
C^{(2)}
\to
C^{(3)}
\to\cdots
\]

For all $n\geq 0$, we obtain $C^{(n+1)}$ from $C^{(n)}$ in steps. We first apply the evaluation functor at $n$ to land in the infinity groupoid $C^{(n)}_n$. While here, we identify an ideal $n$-sphere $\beta_n:S^n\to C^{(n)}_n$. We now transport the following span of $\infty$-groupoids to the category of co-globular $\infty$-groupoids using the left adjoint to the evaluation at $n$ functor, call it $F_n$, and apply the counit to the resulting $F_n(C^{(n)}_n)$, so that our new span has endpoint $C^{(n)}$ rather than $F_n(C^{(n)}_n)$.
\[
D^{n+1}\xleftarrow{j_n}S^n\xrightarrow{\beta_n}C^{(n)}_n
\]
We then define $C^{(n+1)}$ to be the pushout of this induced span. 

We now informally describe how our construction works. Imagine if you would, that you were given a collection of puzzles indexed by the natural numbers that piece together to ultimately build all of the higher dimensional cylinders on the first floor of a two story building. You are tasked with building the puzzles and having the built puzzles on the second floor, with the first floor being played by the category of infinity groupoids and the second floor being played by the category of co-globular Grothendieck $\infty$-groupoids. Suppose we also know that there is an ordering on the puzzles determined by the number of layers need to be put together to complete that puzzle. For our construction, we tackle the puzzle problem by first laying down the border of each puzzle and carrying the border to the second floor, with the border in this case being played by $D_\bullet\coprod D_\bullet$.

\[
\begin{tikzpicture}[
    scale=1.2,
    line cap=round,
    line join=round
]


\begin{scope}[xshift=0cm]

    \node at (0,2.5) {$D^0 \amalg D^0$};

    \fill (-0.6,1.2) circle (2.5pt);
    \fill ( 0.6,1.2) circle (2.5pt);

\end{scope}


\begin{scope}[xshift=4cm]

    \node at (0.6,2.5) {$D^1 \amalg D^1$};

    \draw (0,-0.2) -- (0,1.4);
    \fill (0,-0.2) circle (2.5pt);
    \fill (0,1.4) circle (2.5pt);

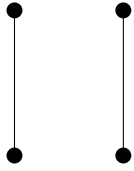
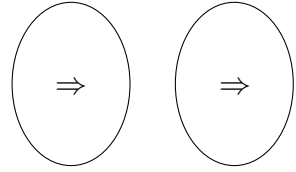
    \draw (1.2,-0.2) -- (1.2,1.4);
    \fill (1.2,-0.2) circle (2.5pt);
    \fill (1.2,1.4) circle (2.5pt);

\end{scope}


\begin{scope}[xshift=9cm]

    \node at (0.9,2.5) {$D^2 \amalg D^2$};

    \draw (0,0.6)
        ellipse [x radius=0.65, y radius=0.9];

    \draw (1.8,0.6)
        ellipse [x radius=0.65, y radius=0.9];

    \node at (0,0.55) {$\Rightarrow$};

    \node at (1.81,0.55) {$\Rightarrow$};

\end{scope}

\end{tikzpicture}
\]

Then once the border is on the second floor, we now identify the next layer of the puzzle to lay down, we build that on the first floor, and then we carry that bit to the second floor and piece it together with the border. 

\[
\begin{tikzpicture}[
    scale=1.2,
    line cap=round,
    line join=round,
    >=stealth
]


\begin{scope}[xshift=0cm]

    \node at (0,2.5) {$C^{(1)}_0$};

    \fill (-0.6,1.2) circle (2.5pt);
    \fill ( 0.6,1.2) circle (2.5pt);

    \draw[->] (-0.55,1.2) -- (0.55,1.2);

\end{scope}


\begin{scope}[xshift=4cm]

    \node at (0.6,2.5) {$C^{(1)}_1$};

    \draw (0,-0.2) -- (0,1.4);
    \fill (0,-0.2) circle (2.5pt);
    \fill (0,1.4) circle (2.5pt);

    \draw (1.2,-0.2) -- (1.2,1.4);
    \fill (1.2,-0.2) circle (2.5pt);
    \fill (1.2,1.4) circle (2.5pt);

    \draw[->] (0,1.4) -- (1.2,1.4);
    \draw[->] (0,-0.2) -- (1.2,-0.2);

\end{scope}


\begin{scope}[xshift=9cm]

    \node at (0.9,2.5) {$C^{(1)}_2$};

    \draw (0,0.6)
        ellipse [x radius=0.65, y radius=0.9];

    \draw (1.8,0.6)
        ellipse [x radius=0.65, y radius=0.9];

    \draw[->] (0,1.5) -- (1.8,1.5);
    \draw[->] (0,-0.3) -- (1.8,-0.3);

    \node at (0,0.65) {$\Rightarrow$};

    \node at (1.81,0.65) {$\Rightarrow$};

\end{scope}

\end{tikzpicture}
\]

We now build the next layer in the exact same fashion.
\[
\begin{tikzpicture}[
    scale=1.2,
    line cap=round,
    line join=round,
    >=stealth
]


\begin{scope}[xshift=0cm]

    \node at (0,2.5) {$C^{(2)}_0$};

    \fill (-0.6,1.2) circle (2.5pt);
    \fill (0.6,1.2) circle (2.5pt);

    \draw[->] (-0.55,1.2) -- (0.55,1.2);

\end{scope}


\begin{scope}[xshift=4cm]

    \node at (0.6,2.5) {$C^{(2)}_1$};

    \draw (0,0) rectangle (1.2,1.6);

    \draw[->] (0,0) -- (1.2,0);
    \draw[->] (0,1.6) -- (1.2,1.6);

    \node at (0.6,0.8) {$\Rightarrow$};

    \fill (0,0) circle (2.5pt);
    \fill (0,1.6) circle (2.5pt);
    \fill (1.2,0) circle (2.5pt);
    \fill (1.2,1.6) circle (2.5pt);

\end{scope}


\begin{scope}[xshift=9cm]

    \node at (0.9,2.5) {$C^{(2)}_2$};

    \draw (0,0.3) -- (1.8,0.3);
    \draw (0,2.1) -- (1.8,2.1);

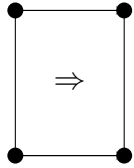
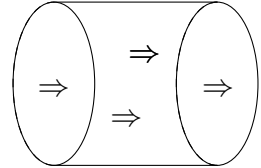
    \draw (0,1.2)
        ellipse [x radius=0.45, y radius=0.9];

    \draw (1.8,1.2)
        ellipse [x radius=0.45, y radius=0.9];

    \draw[dashed]
        (0,0.3) arc (270:90:0.45 and 0.9);

    \draw[dashed]
        (1.8,0.3) arc (270:90:0.45 and 0.9);

    \node at (1,1.5) {$\Rightarrow$};

    \node at (0.8,0.8) {$\Rightarrow$};

    \node at (1,1.5) {$\Rightarrow$};

    \node at (0,1.125) {$\Rightarrow$};

    \node at (1.81,1.125) {$\Rightarrow$};

\end{scope}
\end{tikzpicture}
\]
Our construction is inductive, it builds a co-globular $\infty$-groupoid at every step, and it is built completely using universal properties at every stage in the construction. Additionally, we prove that the construction we provide is isomorphic to the one Lanari provides using formulas that we generate and the existence of a system of pushouts that the the construction of Lanari and the one presented here both present.

In Section 1, we give background on pure colimit sketches and algebraic weak factorization systems. In Section 2, we provide the necessary background on $(\infty,0)$-coherators, the \emph{algebraic coherator}, and Grothedieck $\infty$-groupoids. In the final section, we begin by exploring the left adjoints to the evaluation functors at a specific co-dimension and the \emph{projective generating cofibrations}. We then build a new cylinder object, denoted by $\mathbf{Cyl}_\bullet$, in the category of co-globular infinity groupoids. We conclude by proving the following theorem.
\begin{center}
\textbf{Theorem \ref{iso_to_Lan}}
\emph{There is an isomorphism of of co-globular $\infty$-groupoids $\phi:\mathbf{Cyl}_\bullet\to \mathbf{E}$,
where $\mathbf{E}$ is the cylinder construction given by Lanari.}
\end{center}

\subsection*{Special Thanks from the Author}
I would like to thank my husband for constant support. I would like to thank Simon Henry for all of our back and forth conversations on this topic. 

\subsection*{Remarks about Work from the Author}
The arguments I provide in later sections regarding the algebraic coherator $AC$ work equally well for any $(\infty,0)$-coherator. When relevant, I mention why it works and how. I use the algebraic coherator, exclusively, because it is my favorite coherator and it makes the constructions in this paper algebraic and completely independent of any self-made choices. Moreover, it has been the coherator that I have had the most success working with in Python script. 

\subsection*{Common Notation}
We introduce the following notation, which will be used throughout the remainder of this paper.

\begin{itemize}
    \item $\mathbf{Set}$ denotes the category of sets.
    \item $\mathbf{Top}$ denotes the category of compactly generated weakly Hausdorff spaces.
    \item Given a category $C$, we write $C^2$ for the arrow category.
    \item Given categories $C$ and $D$, we write $[C,D]$ for the category of functors from $C$ to $D$.
\end{itemize}

\section{Background}
We now begin by providing background on two topics that are important enough to the paper to include here; however, the theory they involve is more expansive and we only use bits of that complete theory for this paper.
\subsection*{Pure colimit sketches}
In this short subsection, we provide the definition of pure colimit sketch and state one of its fundamental consequences for the purposes of this paper.
\begin{definition}
A \emph{pure colimit sketch} consists of a small category $\mathcal{S}$ together with a specified collection of colimit co-cones in $\mathcal{S}$. A \emph{model} of $\mathcal{S}$ in a category $\mathcal{E}$ is a functor
\[
X : \mathcal{S}\to\mathcal{E}
\]
that sends each specified colimit co-cone in $\mathcal{S}$ to a colimit co-cone in $\mathcal{E}$. We denote the category of models by
\[
\mathbf{Mod}(\mathcal{S},\mathcal{E}).
\]
\end{definition}

\begin{proposition}\label{pure_colimit_models_colimits}
Let $\mathcal{S}$ be a pure colimit sketch, and let $\mathcal{E}$ be a cocomplete category. Then the inclusion
\[
\mathbf{Mod}(\mathcal{S},\mathcal{E})
\hookrightarrow
[\mathcal{S},\mathcal{E}]
\]
creates colimits. In particular, if $D: J\to \mathbf{Mod}(\mathcal{S},\mathcal{E})$
is a diagram of models, then its colimit is computed point-wise. That is,
\[
\left(\operatorname{colim}_{j\in J}D_j\right)(A)
\cong
\operatorname{colim}_{j\in J}D_j(A)
\]
for every object $A\in\mathcal{S}$.
\end{proposition}

\begin{proof}
Let $D: J\to \operatorname{Mod}(\mathcal{S},\mathcal{E})$
be a diagram. Since $\mathcal{E}$ is cocomplete, the point-wise colimit
\[
M=\operatorname{colim}_{j\in J}D_j
\]
exists in $[\mathcal{S},\mathcal{E}]$. Therefore, we have that
\[
M(A)=\operatorname{colim}_{j\in J}D_j(A)
\]
for every $A\in\mathcal{S}$.
We must now show that $M$ is a model of $\mathcal{S}$. Let $C:I\to\mathcal{S}$
be a distinguished colimit diagram with specified colimit co-cone having vertex $A$. Since each $D_j$ is a model of $\mathcal{S}$, the image of this co-cone under $D_j$ is a colimit co-cone in $\mathcal{E}$. Hence
\[
D_j(A)\cong\operatorname{colim}_{i\in I}D_j(C(i)).
\]
Therefore, we must have that
\begin{align*}
M(A)
&\cong \operatorname{colim}_{j\in J}D_j(A)\\
&\cong \operatorname{colim}_{j\in J}
    \operatorname{colim}_{i\in I}D_j(C(i))\\
&\cong \operatorname{colim}_{i\in I}
    \operatorname{colim}_{j\in J}D_j(C(i))\\
&\cong \operatorname{colim}_{i\in I}M(C(i)).
\end{align*}
The third isomorphism follows from the fact that colimits commute with colimits. Therefore $M$ sends every distinguished colimit co-cone of $\mathcal{S}$ to a colimit co-cone in $\mathcal{E}$. We conclude that $M$ is a model of $\mathcal{S}$.

It follows that the point-wise colimit of the diagram $D$ lies in $\mathbf{Mod}(\mathbf{S},\mathcal{E})$. Since $\operatorname{Mod}(\mathcal{S},\mathcal{E})$ is a full subcategory of $[\mathcal{S},\mathcal{E}]$, this point-wise colimit satisfies the universal property of the colimit in $\mathbf{Mod}(\mathcal{S},\mathcal{E})$. Therefore, the inclusion
\[
\mathbf{Mod}(\mathcal{S},\mathcal{E})
\hookrightarrow
[\mathcal{S},\mathcal{E}]
\]
creates colimits.
\end{proof}

\subsection*{Algebraic weak factorization systems}
We now recall the theory of algebraic weak factorization systems of Garner \cite{Garner2009}.
\begin{definition}
An \emph{algebraic weak factorization system} on a category $C$
is a pair $(L,R)$ consisting of a comonad $L$ and a monad $R$ on the arrow
category $C^{\mathbf{2}}$ such that every morphism
$f: X\to Y$ admits a functorial factorization
\[
X\xrightarrow{Lf}Ef\xrightarrow{Rf}Y
\]
of $f$, where the action of $L$ and $R$ on objects is given by the two
maps in this factorization, and such that the comonad and monad structures
are compatible with the factorization.
\end{definition}

\begin{definition}
An algebraic weak factorization system $(L,R)$ on a category $C$
is said to be \emph{cofibrantly generated} if there exists a set of
morphisms $\mathcal{I}$ of $C$ such that $(L,R)$ is the algebraic weak factorization system generated by
$\mathcal{I}$ via the algebraic small object argument. In this case,
$\mathcal{I}$ is called a set of \emph{generating cofibrations}.
\end{definition}

Let $(L,R)$ be an algebraic weak factorization system on $C$ cofibrantly generated by a set of maps $I$. The
\emph{fibrant replacement} functor
\[
R_I: C\to C
\]
is defined by factoring the terminal map $X\to *$ as
\[
X\xrightarrow{\eta_X}R_I(X)\to *,
\]
where
\[
R_I(X):=E(X\to *)
\qquad\text{and}\qquad
\eta_X:=L(X\to *).
\]
Thus $\eta_X: X\to R_I(X)$ is the induced fibrant replacement map.

\section{Background on Grothendieck infinity groupoids}
We now provide background on the theory of Grothendieck $\infty$-groupoids following the work of
Maltsiniotis~\cite{Malt}, Bourke~\cite{Bo2}, and Ara~\cite{Dim1}.

\subsection*{Basics of globular sets and coherators} 

To begin with, we define globular and co-globular objects in a category $C$. Globular objects provide a combinatorial framework for describing higher-dimensional cells together with their source and target operations. Their
indexing category is the \emph{globe category}.

\begin{notation}
The \emph{globe category} $\mathbf{G}$ is the category
generated by the graph 
\[
\begin{tikzpicture}[node distance=2cm]
    \node (dots) {$\cdots$};
    \node (3) [right of=dots] {$3$};
    \node (2) [right of=3] {$2$};
    \node (1) [right of=2] {$1$};
    \node (0) [right of=1] {$0$};

    \draw[transform canvas={yshift=0.4ex},->]
        (3) to (dots);

            \draw[transform canvas={yshift=-0.4ex},->]
        (3) to  (dots);

    \draw[transform canvas={yshift=0.4ex},->]
        (2) to node[above] {$s_2$} (3);

    \draw[transform canvas={yshift=-0.4ex},->]
        (2) to node[below] {$t_2$} (3);

    \draw[transform canvas={yshift=0.4ex},->]
        (1) to node[above] {$s_1$} (2);

    \draw[transform canvas={yshift=-0.4ex},->]
        (1) to node[below] {$t_1$} (2);

    \draw[transform canvas={yshift=0.4ex},->]
        (0) to node[above] {$s_0$} (1);

    \draw[transform canvas={yshift=-0.4ex},->]
        (0) to node[below] {$t_0$} (1);
\end{tikzpicture}
\]
subject to the relations
\[
s_{n+1}\circ s_n=t_{n+1}\circ s_n,
\qquad
s_{n+1}t_n=t_{n+1}t_n.
\]
\end{notation}

\begin{definition}
A \emph{globular object} in a category $C$ is a functor
\[
X:\mathbf{G}^{\mathbf{op}}\to C.
\]
A \emph{co-globular object} in a category $C$ is a functor
\[
X:\mathbf{G}\to C.
\]
\end{definition}

\begin{example}\label{example_important}
For each $n\geq0$, let
\[
D^n=\{x\in\mathbf{R}^n:\|x\|\leq1\}.
\]
The assignments
\[
\sigma^n(x)=\left(x,\sqrt{1-\|x\|^2}\right),
\qquad
\tau^n(x)=\left(x,-\sqrt{1-\|x\|^2}\right)
\]
define a co-globular object
\[
D:\mathbf{G}\to\mathbf{Top},
\]
where $\sigma^n$ and $\tau^n$ include the upper and lower hemispheres of
$D^{n-1}$ into $D^n$. Consequently every space $X$ determines a globular
object
\[
\mathbf{Top}(D^{(-)},X).
\]
\end{example}

\subsection{Globular Products and Operations}
In addition to an appropriate combinatorial framework, we require certain limits that mimic the Segal condition for complete Segal spaces (see \cite{Rezk2001}) to exist. Additionally, we require operations that allow us to combine the data together to exist. We first define the necessary limits.

\begin{definition}
A \emph{table of dimensions} is an odd-length sequence
\[
\vec n=(n_1,\ldots,n_{2k+1})
\]
of natural numbers such that
\[
n_1>n_2<n_3\cdots>n_{2k}<n_{2k+1}.
\]
\end{definition}

A table of dimensions 
\[
\vec n=(n_1,\ldots,n_{2k+1})
\]
induces a zig-zag diagram in $\mathbf{G}$. Conversely, a zig-zag diagram in $\mathbf{G}$ induces a table of dimensions. We naturally assign a notion of height for every table of dimensions. 

\begin{definition}
The \emph{height} of a table of dimensions $\vec n$ is
\[
\mathbf{hgt}(\vec n)=\max_i n_i.
\]
\end{definition}

Moreover, we provide a naming convention for the limit (or colimit) of the induced zig-zag diagram arising from a table of dimensions.

\begin{definition}
Let $A:\mathbf{G}^{op}\to C$ be a globular object and let $\vec n$ be a
table of dimensions. The limit of the induced zig-zag diagram in $C$, when it
exists, is called the \emph{globular product}.

Let $A:\mathbf{G}\to C$ be a co-globular object and let $\vec n$ be a
table of dimensions. The colimit of the induced zig-zag diagram in $C$, when it
exists, is called the \emph{globular sum}.
\end{definition}

Another interpretation, other than that of the Segal condition, is that globular products play the role for globular theories that finite products play for Lawvere theories; they capture and encode the way in which the operations interact with each other.

\begin{definition}
The category $\Theta_0$ has tables of dimensions as objects and
\[
\Theta_0(\vec n,\vec m)
=
[\mathbf{G}^{op},\mathbf{Set}]
(Y(\vec n),Y(\vec m)),
\]
where $Y$ denotes the Yoneda embedding.
\end{definition}

Given this definition, the category $\Theta_0^\mathbf{op}$ comes equipped with the following universal property.

\begin{lemma}[Lemma 2.1 of \cite{Bo2}]
There is a functor
\[
D:\mathbf{G}^{op}\to\Theta_0^{op}
\]
that is universal among globular-product-preserving extensions. Explicitly,
if
\[
A:\mathbf{G}^{\mathbf{op}}\to C
\]
admits globular products, then there exists an essentially unique
globular-product-preserving extension
\[
A':\Theta_0^{op}\to C.
\]
Moreover, $A'(\vec n)$ is the globular product associated to $\vec n$.
\end{lemma}

We now give the theory for admissible pairs, lifts, and contractibility. Viable options for higher category theory always come equipped with a notion of contractibility for the operations.

\begin{definition}
Let $A$ be a globular object. An \emph{admissible pair} of dimension $n$
consists of parallel maps
\[
\begin{tikzpicture}[node distance=2cm]
    \node (A) {$X$};
    \node (B) [right of=A] {$A(n)$};
    \draw[transform canvas={yshift=0.3ex},->]
        (A) to node[above=3] {$f$} (B);
    \draw[transform canvas={yshift=-0.3ex},->,swap]
        (A) to node[below=3] {$g$} (B);
\end{tikzpicture}
\]
where $X$ is a globular product of height at most $n+1$ and either $n=0$ or
\[
A(s)\circ f=A(s)\circ g\qquad A(t)\circ f=A(t)\circ g.
\]
\end{definition}

\begin{definition}
A \emph{lift} of an admissible pair $(f,g)$ is a morphism
\[
\delta_{f,g}:X\to A(n+1)
\]
such that
\[
A(s)\circ \delta_{f,g}=f\qquad  A(t)\circ \delta_{f,g}=g.
\]
\end{definition}

\begin{definition}
A globular object is \emph{contractible} if every admissible pair admits a
lift. A globular theory is contractible when regarded as a globular object.
\end{definition}

In addition to being contractible, we require a theory that is obtained freely. We require a freeness condition since it is what allows us to express the necessary axioms for fully weak $\infty$-groupoids on models. 

\begin{definition}
An $(\infty,0)$\emph{-coherator} is a contractible globular theory obtained as
the colimit of a sequence
\[
\Theta_0^{op}=C_0
\to
C_1
\to
C_2
\to\cdots,
\]
where each morphism is obtained by freely adjoining lifts for a chosen collection of admissible
pairs at every step.
\end{definition}

\begin{notation}
If $C$ is an $(\infty,0)$-coherator, we write
\[
\infty\mathbf{Gpd}_C
=
\mathbf{Mod}_{\Theta_0^{op}}(C)
\]
for the category of globular-product-preserving functors $X:C\to\mathbf{Set}$. Its objects are called \emph{Grothendieck $\infty$-groupoids}.
\end{notation}

\subsection*{The algebraic coherator}

In this subsection, we construct the algebraic weak factorization system that will generate a theory we use for the rest of the paper. The generating maps encode the operations by freely adjoining lifts to admissible pairs. Applying Garner's algebraic small object argument then produces a fibrant replacement monad whose fibrant replacement of the initial globular theory is an $(\infty,0)$-coherator. Moreover, it allows us to obtain an algebraic choice of compositions and whiskering operations, so that everything that we do in this paper is done in a completely algebraic fashion. We begin by defining our required presentable objects and generating maps.
\begin{definition}\label{spheres_in_inf_gpds}
Let $\vec{p}$ be a table of dimensions and let $k\geq0$ satisfy
\[
\mathbf{hgt}(\vec{p})\leq k+1.
\]
The \emph{$(\vec{p},k)$-sphere} is the globular theory
\[
S_{\vec{p},k}
\]
obtained from $\Theta_0^\mathbf{op}$ by freely adjoining an admissible pair
\[
\begin{tikzpicture}[node distance=2cm]
    \node (A) {$\vec{p}$};
    \node (B) [right of=A] {$k$};

    \draw[transform canvas={yshift=0.4ex},->]
        (A) to node[above] {$f$} (B);

    \draw[transform canvas={yshift=-0.4ex},->]
        (A) to node[below] {$g$} (B);
\end{tikzpicture}
\]
subject to the relations
\[
s\circ f=s\circ g
\qquad\text{and}\qquad
t\circ f=t\circ g
\]
whenever $k\geq1$. Equivalently, $S_{\vec{p},k}$ is the free globular theory containing a distinguished admissible pair of dimension $k$ with domain $\vec{p}$. There is a canonical inclusion
\[
\Theta_0^\mathbf{op}\hookrightarrow S_{\vec{p},k}.
\]
\end{definition}

\begin{definition}\label{disks_in_inf_gpds}
The \emph{$(\vec{p},k)$-disk}
\[
D_{\vec{p},k}
\]
is obtained from $S_{\vec{p},k}$ by freely adjoining a lift
\[
\delta_{f,g}:\vec{p}\to k+1
\]
of the distinguished admissible pair so that the diagram
\[
\begin{tikzpicture}[node distance=2cm]
    \node (A) {$\vec{p}$};
    \node (B) [right of=A] {$k$};
    \node (C) [above of=B] {$k+1$};

    \draw[transform canvas={yshift=0.4ex},->]
        (A) to node[above] {$f$} (B);

    \draw[transform canvas={yshift=-0.4ex},->]
        (A) to node[below] {$g$} (B);

    \draw[transform canvas={xshift=-0.4ex},->]
        (C) to node[left] {$s$} (B);

    \draw[transform canvas={xshift=0.4ex},->]
        (C) to node[right] {$t$} (B);

    \draw[->]
        (A) to node[left] {$\delta_{f,g}$} (C);
\end{tikzpicture}
\]
commutes. Thus $D_{\vec{p},k}$ is the free globular theory obtained by specifying a filler for the distinguished admissible pair. It is equipped with the canonical inclusion
\[
\Theta_0^\mathbf{op}\hookrightarrow D_{\vec{p},k}.
\]
\end{definition}

\begin{notation}\label{incl_of_sphere_into_disks}
For every table of dimensions $\vec{p}$ and every $k\geq0$, there is a canonical inclusion
\[
j_{\vec{p},k}:S_{\vec{p},k}\to D_{\vec{p},k}.
\]
We write
\[
I=
\left\{
j_{\vec{p},k}
:
S_{\vec{p},k}\to D_{\vec{p},k}
\;\middle|\;
\vec{p}\in\mathbf{ob}(\Theta_0^\mathbf{op}),\;
k\geq0
\right\}.
\]
\end{notation}

\begin{lemma}\label{admiss_1}
Each object $S_{\vec{p},k}$ and $D_{\vec{p},k}$ is presentable in
$\mathbf{Th}_{\Theta_0^\mathbf{op}}$.
Moreover, the set $I$ is admissible for Garner's algebraic small object argument.
\end{lemma}

\begin{proof}
This follows from Subsection~3.11 and Lemma~3.12 of \cite{Malt}.
\end{proof}

\begin{notation}
Let
\[
(L_I,E_I,R_I,\delta_I)
\]
denote the algebraic weak factorization system on
$\mathbf{Th}_{\Theta_0^\mathbf{op}}$
cofibrantly generated by the set $I$.
\end{notation}

\begin{lemma}
Let
\[
J^{AC}:\Theta_0^\mathbf{op}\to AC
\]
be the fibrant replacement of the initial object
\[
\mathrm{id}_{\Theta_0^\mathbf{op}}
\]
with respect to the algebraic weak factorization system
$(L_I,E_I,R_I,\delta_I)$.
Then $AC$ is an $(\infty,0)$-coherator.
\end{lemma}

\begin{proof}
The proof is obtained by adapting the argument of Theorem 3.14 of \cite{Malt}.
\end{proof}

We refer to $AC$ as the \emph{algebraic coherator}. From now on, we write
\[
\infty\mathbf{Gpd}:=\mathbf{Mod}(AC)
\]
and mean a model over $AC$ whenever we call upon an $\infty$-groupoid.

We finish this subsection by listing operations that we will use in a future section to build the cylinder object.
\begin{notation}[Algebraic Composition and Whiskering]\label{algeb_whisk_and_comp}
For all $n\geq 1$, the algebraic small object argument generates a choice of $n$-fold composition operation along an $(n-1)$-boundary, that we denote by 
\[
c_{n,0}:(n,n-1,n)\to n,
\]
as the specified choice of lift for the following admissible pair, where $\epsilon_1$ and $\epsilon_2$ are the projections onto the first and second copy of $1$, respectively.
\[
\begin{tikzpicture}[node distance=3cm]
    \node (A) {$(n,n-1,n)$};
    \node (B) [right of=A] {$n$};
    \draw[transform canvas={yshift=0.3ex},->]
        (A) to node[above=3] {$s\circ \epsilon_1$} (B);
    \draw[transform canvas={yshift=-0.3ex},->,swap]
        (A) to node[below=3] {$t\circ \epsilon_2$} (B);
\end{tikzpicture}
\]
For all $n\geq 1$ and $i\geq 0$, we define the \emph{left and right whiskering operations}, denoted by
\[
c^l_{n,i+1}:(n+i+1,n-1,n)\to n+i+1\qquad c^r_{n,i}:(n,n-1,n+i+1)\to n+i+1,
\]
as the specified choice of lift generated by the algebraic small object argument for the following admissible pairs, respectively.
\[
\begin{tikzpicture}[node distance=5cm]
    \node (A) {$(n+i+1,n-1,n)$};
    \node (B) [right of=A] {$n+i$};
    \draw[transform canvas={yshift=0.3ex},->]
        (A) to node[above=3] {$c^l_{n,i}\circ (s,\mathbf{id}_n)$} (B);
    \draw[transform canvas={yshift=-0.3ex},->,swap]
        (A) to node[below=3] {$c^l_{n,i}\circ (t,\mathbf{id}_n)$} (B);
\end{tikzpicture}
\]

\[
\begin{tikzpicture}[node distance=5cm]
    \node (A) {$(n,n-1,n+i+1)$};
    \node (B) [right of=A] {$n+i$};
    \draw[transform canvas={yshift=0.3ex},->]
        (A) to node[above=3] {$c^r_{n,i}\circ (\mathbf{id}_n,s)$} (B);
    \draw[transform canvas={yshift=-0.3ex},->,swap]
        (A) to node[below=3] {$c^r_{n,i}\circ (\mathbf{id}_n,t)$} (B);
\end{tikzpicture}
\]
\end{notation}

We call these operations algebraic because they were specially provided to us by the algebraic small object argument rather than being chosen. For abuse of notation and to save space with writing in a later section, we write 
\[
c^l_{n,0}=c^r_{n,0}=c_{n,0}
\]
for all $n\geq 1$.

\section{Co-globular infinity groupoids and cylinder objects}
In this section, we construct the cylinder object in the category of co-globular objects of infinity groupoids. We begin with defining the projective generators in the category of co-globular infinity groupoids.

\subsection*{Projective Generators and Extensions}
We now define important adjunctions that we use to construct the cylinder object. First, to provide a more compact presentation of the present mathematics, we set
\[
\mathcal M
=
[\mathbb G,\infty\mathbf{Gpd}]
\]
denote the forgetful functor and
\[
\operatorname{ev}_k:
[\mathbb G,\infty\mathbf{Gpd}]
\to
\infty\mathbf{Gpd}
\]
denote evaluation at $k$ for all $k\geq0$. 

For all $k\geq 0$, $U_k$ has a left adjoint $F_k$ defined by sending an infinity groupoid $X$ to an object of $\mathcal{M}$, denote it by $F_k(X)$, given by 
\[
F_k(X)_m=\begin{cases}
    \emptyset\qquad \text{if }m<k,\\
    X\qquad \text{if }m= k,\\
    X\coprod X\qquad \text{if }m> k
\end{cases}
\]
on the objects of $\mathbb{G}$. On morphisms, we have that 
\[
F_k(X)(t)=F_k(X)(s):\emptyset\to F_{k}(X)(i+1)
\]
is the initial map for $0\leq i< k$,
\[
F_k(X)(s),F_k(X)(t):F_{k}(X)(k)\rightrightarrows F_{k}(X)(k+1)
\]
are the coproduct inclusion maps of $X$ into the first and second copy, respectively,
and 
\[
\mathbf{id}_{X\coprod X}=F_k(X)(s)=F_k(X)(s):F_{k}(X)(i)\to F_{k}(X)(i+1)
\]
for $i>k$. We now prove that $F_k$ is left adjoint to $U_k$ for all $k\geq 0$.

\begin{proposition}\label{left_adjoint_to_evaluation}
The functor $F_k$ is left adjoint to $U_k$ for all $k\geq 0$.
\end{proposition}

\begin{proof}
Fix $k\geq 0$. Before we go further, we remark that 
\[
U_k\circ F_k=\mathbf{id}_{\infty\mathbf{Gpd}}.
\]
Now define a natural transformation 
\[
\eta:\mathbf{id}_{\infty\mathbf{Gpd}}\to U_kF_k
\]
by setting its $X$ component to be identity for all $\infty$-groupoids $X$, i.e., $\eta^X=\mathbf{id}_X$. On the other hand, define a natural transformation 
\[
\epsilon:F_kU_k\to \mathbf{id}_{\mathcal{M}}
\]
by setting its $X$ component to be defined by 
\[
\epsilon^X_n=\begin{cases}
    !:\emptyset\to X_n\qquad\text{if }n<k,\\
    \mathbf{id}_{X_k}\qquad \text{if }n=k, \\
    (s^{n-k},t^{n-k})\qquad \text{if }n>k.
\end{cases}
\]
This is now an adjunction because having that 
\[
U_k\circ F_k=\mathbf{id}_{\infty\mathbf{Gpd}}
\]
forces all the maps involved in the two triangle equalities to be identities.
\end{proof}

The following lemmas hold as a result of $F_k$ being a left adjoint.

\begin{lemma}\label{Yoneda_properties_1}
A map
\[
F_k(D^r)\to X
\]
is equivalent to the choice of an $r$-cell of the $\infty$-groupoid
$X_k$ for all $k,r\geq 0$.
\end{lemma}

\begin{proof}
By adjunction and the Yoneda lemma, we have the following chain of bijections.
\[
\mathcal{M}(F_k(D^r),X)\cong \infty\mathbf{Gpd}(D^r,X_k)\cong (X_k)_r
\]
This concludes our proof.
\end{proof}

\begin{lemma}\label{Yoneda_properties_2}
A map
\[
F_k(S^{r-1})\to X
\]
is equivalent to the the choice of a pair of parallel $(r-1)$-cells of
$X_k$ for all $k\geq 0$ and $r\geq 1$.
\end{lemma}

\begin{proof}
By adjunction and the Yoneda lemma, we have the following bijection
\[
\mathcal{M}(F_k(S^{r-1}),X)\cong \infty\mathbf{Gpd}(S^{r-1},X_k)
\]
where the last set is the set of parallel $(r-1)$-cells of
$X_k$.
\end{proof}

\begin{lemma}\label{disk_sphere_projective_calculation}
  Given $r,k\geq 0$, an object $X$ of $\mathcal{M}$, and a map $f:F_{k}(S^{r-1})\to X$, the corresponding pushout
\[
\begin{tikzpicture}
\node (A) at (0,2) {$F_k(S^{r-1})$};
\node (B) at (6,2) {$F_k(D^r)$};
\node (C) at (0,0) {$X$};
\node (D) at (6,0) {$X'$};

\draw[->] (A) -- node[above] {} (B);
\draw[->] (A) -- (C);
\draw[->] (B) -- (D);
\draw[->] (C) -- (D);
\end{tikzpicture}
\]
freely adjoins an $r$-cell to $X_k$ having the prescribed pair of parallel
$(r-1)$-cells as its source and target.
\end{lemma}

\begin{proof}
The attaching map specifies the source and target of the
$r$-cell of $X_k$ which is freely adjoined in the pushout.
\end{proof}

We now need to determine what happens with a map at the other levels. Given a co-globular $\infty$-groupoids and a map $f:F_k(A)\to X$, we have that
\[
f_{k+1}\circ \iota_1=\sigma\circ f_{k}\qquad \text{and}\qquad f_{k+1}\circ \iota_2=\tau\circ f_{k},
\]
where $\iota_1$ and $\iota_2$ are the canonical coproduct structure maps. Moreover, we have that 
\[
f_{m+1}=\sigma\circ f_m=\tau\circ f_m
\]
for all $m>k$. Moreover, the morphism $f_m$ is the initial map for $m<k$. This says that the data determined by the map is rigid in higher dimensions and non-existent in lower dimensions.

We now make more sense of this by explaining what pushing out along a \emph{projective generating cofibration} does.

\begin{definition}
The set of \emph{projective generating cofibrations} is given by 
\[
\{F_k(S^{r-1})\xrightarrow{F_k(j_r)}F_k(D^{r}):k,r\geq 0\}
\]

\end{definition}
We recall from Proposition \ref{pure_colimit_models_colimits} that the colimits of models over pure colimit sketches are computed point-wise. Therefore, for a projective generating cofibration
\[
F_k(S^n)\to F_k(D^{n+1}),
\]
the cell attachment at co-level $k$ is the ordinary boundary inclusions of $\infty$-groupoids.

\[
S^n\to D^{n+1}
\]

Since $F_k(S^n)_m=S^n\coprod S^n$ and $F_k(D^{n+1})_m=D^{n+1}\coprod D^{n+1}$ for every $m> k$, we obtain the following pushout
\[
\begin{tikzpicture}
\node (A) at (0,2) {$S^n\coprod S^n$};
\node (B) at (5,2) {$D^{n+1}\coprod D^{n+1}$};
\node (C) at (0,0) {$Z_{k,n}$};
\node (D) at (5,0) {$Z'_{k,n}$};

\draw[->] (A) -- (B);
\draw[->] (A) -- (C);
\draw[->] (B) -- (D);
\draw[->] (C) -- (D);
\end{tikzpicture}
\]
at every co-level $m> k$ due to Proposition \ref{pure_colimit_models_colimits}. On the other hand,  no new data is added at co-levels lower than $k$ due to Proposition \ref{pure_colimit_models_colimits}.

\subsection*{Lanari's Cylinders}
We now provide the construction of the cylinders of Lanari. Write $\Sigma$ for the suspension functor of Lanari.  Write $\mathbf{E}_0=D^1$. By induction, define $\mathbf{E}_n$ as the colimit of the following diagram for all $n\geq 1$.

\[
\begin{tikzpicture}[scale=1.1]
    \node (A) at (0,3) {$D^n\amalg_{D^0}D^1$};
    \node (B) at (-2,2) {$D^n$};
    \node (C) at (0,1) {$\Sigma\mathbf{E}_{n-1}$};
    \node (D) at (-2,0) {$D^n$};
    \node (E) at (0,-1) {$D^1\amalg_{D^0}D^n$};

    \draw[->] (B) -- node[above left] {$c^l_{1,n-1}$} (A);
    \draw[->] (B) -- node[above right] {$\Sigma(\iota_0)$} (C);

    \draw[->] (D) -- node[above left] {$\Sigma(\iota_1)$} (C);
    \draw[->] (D) -- node[below left] {$c^r_{1,n-1}$} (E);
\end{tikzpicture}
\]
In this diagram, \[
c^l_{1,n-1}:D^n\to D^1\coprod_{D^0}D^n
\]
and
\[
c^r_{1,n-1}:D^n\to D^n\coprod_{D^0}D^1
\]
are the induced whiskering operations from Notation \ref{algeb_whisk_and_comp} and the maps
\[
\iota_0,\iota_1:D^{n-1}\to \mathbf{E}_{n-1}
\]
are the following composites for all $n\geq 1$.
\[
D^{n-1}\xrightarrow{}D^{n-1}\coprod_{D^0}D^1\xrightarrow{}\mathbf{E}_{n-1}
\]
\[
D^{n-1}\xrightarrow{}D^1\coprod_{D^0}D^{n-1}\xrightarrow{}\mathbf{E}_{n-1}
\]
The maps on the right for both composites are the inclusion into the colimit. We now obtain the coglobular structure by repeating the argument given by Lanari.
\[
\mathbf{E}:\mathbf{G}\to \infty\mathbf{Gpd}
\]

\subsection*{The Cylinder Object Re-realized}
We now define the \emph{canonical cylinder object} for the category of co-globular $\infty$-groupoids. Before we begin, we set the following notation for the \emph{algebraic whiskers}.

\begin{notation}[Fixing Algebraic Whiskers]\label{alg_whiskering}
 For all $n\geq 1$ and $i\geq 0$, we induce operations
 \[
 c^l_{n,i}:D^{n+i}\to D^{n+i}\amalg_{D^{n-1}}D^n\qquad\text{and}\qquad c^r_{n,i}: D^{n+i}\to D^n\amalg_{D^{n-1}}D^{n+i},
 \]
 for left and right whiskerings as the image under the Yoneda embedding of the opposite of the algebraic composition and whiskering operations of Notation \ref{algeb_whisk_and_comp}.
\end{notation}

We make the following remark before proceeding.

\begin{remark}
Lanari picked whiskering maps in \cite{Lanari2020} in the special case that $n=1$. We require these more general compositions and whiskering operations for our construction. Moreover, the whiskers we use are algebraically chosen by the algebraic small object argument and we have made zero choices.
We will remark any choice of $(\infty,0)$-coherators and choice of whiskering maps that are chosen in such a way that the source and target maps are respected will satisfy the argument of this final section.
\end{remark}

We now inductively build the cylinder object. Consider the coproduct of two copies of the co-globular object induced by the Yoneda embedding, i.e., $D_\bullet\amalg D_\bullet$, and set 
\[
C^{(0)}=D_\bullet\amalg D_\bullet.
\]

Define $\lambda^{(0)}_k:D^k\to D^k\coprod D^k$ and $\rho^{(0)}_k:D^k\to D^k\coprod D^k$ to be the inclusion into the first and second factor of the coproduct, respectively. The pair $(\lambda^{(0)}_0,\rho^{(0)}_0)$ is parallel and thus consitutes a map $\beta_0:S^0\to C^{(0)}_0$. By Lemma \ref{Yoneda_properties_2}, we induce a map of the following form.
\[
b_0:F_0(S^0)\to D_\bullet \coprod D_\bullet
\]
We now form the pushout
\begin{equation}\label{initial_pushout}
\begin{tikzpicture}
\node (A) at (0,1.5) {$F_0(S^0)$};
\node (B) at (3,1.5) {$C^{(0)}$};
\node (C) at (0,0) {$F_0(D^1)$};
\node (D) at (3,0) {$C^{(1)}$};

\draw[->] (A) to node[above]{$b_0$} (B);
\draw[->] (A) -- (C);
\draw[->] (B) to node[right]{$J_0$} (D);
\draw[->] (C) to node[below]{$a_1$} (D);
\end{tikzpicture}
\end{equation}
where $C^{(1)}=(D_\bullet\amalg D_\bullet)\amalg_{F_0(S^0)}F_0(D^1)$. A quick calculation shows that we have that
\[
C^{(1)}_k =
\begin{cases}
D^1
    & \text{when } k = 0 \\[8pt]
\displaystyle
(D^k \amalg D^k)
\amalg_{S^0\coprod S^0}
(D^1 \amalg D^1)
    & \text{when } k \geq 1
\end{cases}
\]
Let $\overline{a}^{\sigma,k}_1$ and $\overline{a}^{\tau,k}_1$ be the following composites for all $k\geq 1$.
\[
D^1
\xrightarrow{(a_1)_0}
C^{(1)}_0\xrightarrow{\sigma^{k}_*}C^{(1)}_k
\]

\[
D^1
\xrightarrow{(a_1)_0}
C^{(1)}_0\xrightarrow{\tau^{k}_*}C^{(1)}_k
\]
Now define two maps $\lambda^{(1)}_k,\rho^{(1)}_k:D^k\rightrightarrows C^{(1)}_k$ as the following composites.
\[
D^k
\xrightarrow{c^l_{1,k-1}}
D^k\amalg_{D^0}D^1\xrightarrow{((J_0)_1\circ\lambda^{(0)}_k,\overline{a}^{\tau,k}_1)}C^{(1)}_k
\]
\[
D^k
\xrightarrow{c^r_{1,k-1}}
D^1\amalg_{D^0}D^k\xrightarrow{(\overline{a}^{\sigma,k}_1,(J_0)_1\circ\rho^{(0)}_k)}C^{(1)}_k
\]

For $k=1$, we claim that $(\lambda^{(1)}_1,\rho^{(1)}_1)$ is a
parallel pair. We verify this by pre-composing with the source and
target maps $\sigma,\tau:D^0\to D^1$. By the defining source and target properties of the algebraic
whiskering operations, pre-composition of the left whiskering with
$\sigma$ selects the source of its left factor, while pre-composition
with $\tau$ selects the target of its right factor. Thus
\[
\lambda^{(1)}_1\circ\sigma
=
(J_0)_1\circ\lambda^{(0)}_1\circ\sigma
\]
and
\[
\lambda^{(1)}_1\circ\tau
=
\overline{a}^{\tau,1}_1\circ\tau.
\]

Similarly, pre-composition of the right whiskering with $\sigma$
selects the source of its left factor, while pre-composition with
$\tau$ selects the target of its right factor. Hence
\[
\rho^{(1)}_1\circ\sigma
=
\overline{a}^{\sigma,1}_1\circ\sigma
\]
and
\[
\rho^{(1)}_1\circ\tau
=
(J_0)_1\circ\rho^{(0)}_1\circ\tau.
\]

By the definition of $\overline{a}^{\sigma,1}_1$ and
$\overline{a}^{\tau,1}_1$, together with the pushout defining
$C^{(1)}$, the source of the newly attached $1$-cell is identified
with the source of the first copy of $D^1$, and its target is
identified with the target of the second copy of $D^1$. Therefore
\[
\overline{a}^{\sigma,1}_1\circ\sigma
=
(J_0)_1\circ\lambda^{(0)}_1\circ\sigma
\]
and
\[
\overline{a}^{\tau,1}_1\circ\tau
=
(J_0)_1\circ\rho^{(0)}_1\circ\tau.
\]

It follows that
\[
\lambda^{(1)}_1\circ\sigma
=
\rho^{(1)}_1\circ\sigma
\]
and
\[
\lambda^{(1)}_1\circ\tau
=
\rho^{(1)}_1\circ\tau.
\]

Therefore the pair $(\lambda^{(1)}_1,\rho^{(1)}_1)$ is parallel and constitutes a map $\beta_1:S^1\to C^{(1)}_1$. By Lemma \ref{Yoneda_properties_2}, we induce a map of the following form.
\[
b_1:F_1(S^1)\to C^{(1)}
\]

Before we proceed, we take a break to unpack the data for the reader. Before we pushout along the next projective generating cofibration,  the first three entries of $C^{(1)}$ look like the following.

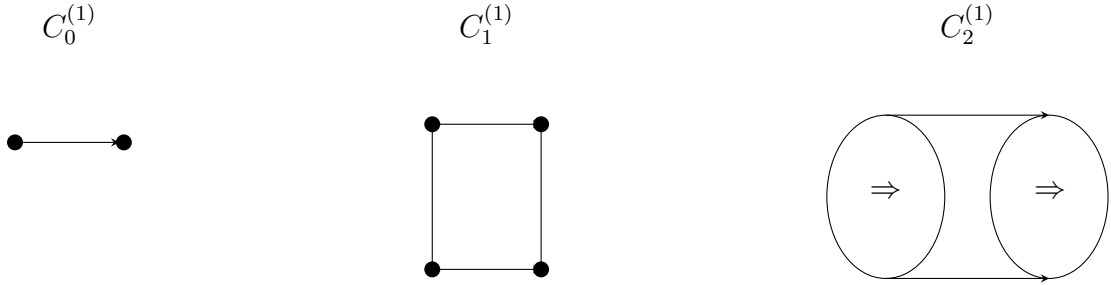
\[
\begin{tikzpicture}[
    scale=1.2,
    line cap=round,
    line join=round,
    >=stealth
]


\begin{scope}[xshift=0cm]

    \node at (0,2.5) {$C^{(1)}_0$};

    \fill (-0.6,1.2) circle (2.5pt);
    \fill ( 0.6,1.2) circle (2.5pt);

    \draw[->] (-0.55,1.2) -- (0.55,1.2);

\end{scope}


\begin{scope}[xshift=4cm]

    \node at (0.6,2.5) {$C^{(1)}_1$};

    \draw (0,-0.2) -- (0,1.4);
    \fill (0,-0.2) circle (2.5pt);
    \fill (0,1.4) circle (2.5pt);

    \draw (1.2,-0.2) -- (1.2,1.4);
    \fill (1.2,-0.2) circle (2.5pt);
    \fill (1.2,1.4) circle (2.5pt);

    \draw[->] (0,1.4) -- (1.2,1.4);
    \draw[->] (0,-0.2) -- (1.2,-0.2);

\end{scope}


\begin{scope}[xshift=9cm]

    \node at (0.9,2.5) {$C^{(1)}_2$};

    \draw (0,0.6)
        ellipse [x radius=0.65, y radius=0.9];

    \draw (1.8,0.6)
        ellipse [x radius=0.65, y radius=0.9];

    \draw[->] (0,1.5) -- (1.8,1.5);
    \draw[->] (0,-0.3) -- (1.8,-0.3);

    \node at (0,0.65) {$\Rightarrow$};

    \node at (1.81,0.65) {$\Rightarrow$};

\end{scope}

\end{tikzpicture}
\]

The maps $\lambda^{(1)}_k$ capture the composite of the $k$-cell on the right and the $1$-cell on the bottom for all $k\geq 0$. On the other hand, $\rho^{(1)}_k$ captures the composite of the $1$-cell on the top and the $k$-cell on the right for all $k\geq 0$. Another thing to remark is that the way we build our cylinder object, $C^{(k)}_0=C^{(1)}_0$ for all $k\geq 1$.

Now form the pushout
\[
\begin{tikzpicture}
\node (A) at (0,1.5) {$F_1(S^1)$};
\node (B) at (3,1.5) {$C^{(1)}$};
\node (C) at (0,0) {$F_1(D^2)$};
\node (D) at (3,0) {$C^{(2)}$};

\draw[->] (A) to node[above]{$b_1$} (B);
\draw[->] (A) -- (C);
\draw[->] (B) to node[right]{$J_1$} (D);
\draw[->] (C) to node[below]{$a_2$} (D);
\end{tikzpicture}
\]

We have that
\[
C^{(2)}_k=
\begin{cases}
C^{(1)}_0, & k=0,\\[2pt]
C^{(1)}_1\amalg_{S^1}D^2, & k=1,\\[2pt]
C^{(1)}_k \amalg_{S^1\amalg S^1}(D^2\amalg D^2), & k\geq2.
\end{cases}
\]
and the first three entries of $C^{(2)}$ look like the following.
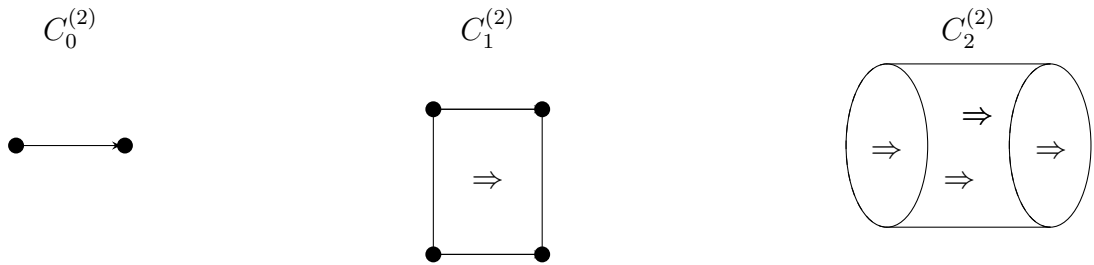
\[
\begin{tikzpicture}[
    scale=1.2,
    line cap=round,
    line join=round,
    >=stealth
]


\begin{scope}[xshift=0cm]

    \node at (0,2.5) {$C^{(2)}_0$};

    \fill (-0.6,1.2) circle (2.5pt);
    \fill (0.6,1.2) circle (2.5pt);

    \draw[->] (-0.55,1.2) -- (0.55,1.2);

\end{scope}


\begin{scope}[xshift=4cm]

    \node at (0.6,2.5) {$C^{(2)}_1$};

    \draw (0,0) rectangle (1.2,1.6);

    \draw[->] (0,0) -- (1.2,0);
    \draw[->] (0,1.6) -- (1.2,1.6);

    \node at (0.6,0.8) {$\Rightarrow$};

    \fill (0,0) circle (2.5pt);
    \fill (0,1.6) circle (2.5pt);
    \fill (1.2,0) circle (2.5pt);
    \fill (1.2,1.6) circle (2.5pt);

\end{scope}


\begin{scope}[xshift=9cm]

    \node at (0.9,2.5) {$C^{(2)}_2$};

    \draw (0,0.3) -- (1.8,0.3);
    \draw (0,2.1) -- (1.8,2.1);

    \draw (0,1.2)
        ellipse [x radius=0.45, y radius=0.9];

    \draw (1.8,1.2)
        ellipse [x radius=0.45, y radius=0.9];

    \draw[dashed]
        (0,0.3) arc (270:90:0.45 and 0.9);

    \draw[dashed]
        (1.8,0.3) arc (270:90:0.45 and 0.9);

    \node at (1,1.5) {$\Rightarrow$};

    \node at (0.8,0.8) {$\Rightarrow$};

    \node at (1,1.5) {$\Rightarrow$};

    \node at (0,1.125) {$\Rightarrow$};

    \node at (1.81,1.125) {$\Rightarrow$};

\end{scope}
\end{tikzpicture}
\]
We now remark that the way we build our cylinder object, we will be forced to have that $C^{(k)}_2=C^{(1)}_2$ for all $k\geq 2$. Moreover, we have our choice of whiskerings from Notation \ref{alg_whiskering}, we have maps that capture the composition of the left $2$-cell along the top $1$-cell and the $1$-cylinder out in front. Therefore we can form the composite to form the target of the $3$-cell that will be added along when we pushout along the next projective generator. Similarly, we have such maps that allow us to form the source of a $3$-cell. Moreover, when we pushout along the next projective generator to form $C^{(3)}$, we will have that $C^{(k)}_2=C^{(3)}_2$ for all $k\geq 3$.

Therefore the way this inductively is going to work is that we will add a single $(n+1)$-cell to $C^{(n)}_n$ when we pushout along the appropriate projective generator to form $C^{(n+1)}_n$. At this point, we will have that $C^{(k)}_n=C^{(n+1)}_n$ for all $k\geq n+1$. Moreover, we add two copies of that single $(n+1)$-cell to $C^{(n)}_k$ to form $C^{(n+1)}_k$ for all $k>n$. The $(n+1)$-cell added to form $C^{(n+1)}_n$ will have source and target completely determined by the maps added by pushing out along the previous projective generators and the data of $\lambda^{(n)}_{n+1}$ and $\rho^{(n)}_{n+1}$. Finally, we will have a formula for the entire co-globular object at every step.

Now suppose that, for some $n\geq 1$, the constructions
$C^{(m)}$, $\lambda^{(m)}_k$, and $\rho^{(m)}_k$ have been made for
all $m<n$ and all $k\geq m$, with $\lambda^{(m)}_m$ and
$\rho^{(m)}_m$ parallel. Additionally, we require that the following formula holds. 
\[
C^{(m)}_t=
\begin{cases}
C^{(t+1)}_t, & t<m-1,\\[2pt]
C^{(m-1)}_{m-1}\amalg_{S^{m-1}}D^{m}, & t=m-1,\\[2pt]
C^{(m-1)}_t\amalg_{S^{m-1}\amalg S^{m-1}}
(D^{m}\amalg D^{m}), & t\geq m
\end{cases}
\]

By the inductive hypothesis, $\lambda^{(n-1)}_{n-1}$ and $\rho^{(n-1)}_{n-1}$ are parallel.
Therefore they determine a map
\[
\beta_{n-1}:S^{n-1}\to C^{(n-1)}_{n-1}.
\]
By Lemma \ref{Yoneda_properties_2}, this induces a map
\[
b_{n-1}:F_{n-1}(S^{n-1})\to C^{(n-1)}.
\]

We form the pushout
\begin{equation}\label{inductive_pushout}
\begin{tikzpicture}
\node (A) at (0,1.5) {$F_{n-1}(S^{n-1})$};
\node (B) at (3,1.5) {$C^{(n-1)}$};
\node (C) at (0,0) {$F_{n-1}(D^n)$};
\node (D) at (3,0) {$C^{(n)}$};

\draw[->] (A) to node[above]{$b_{n-1}$} (B);
\draw[->] (A) -- (C);
\draw[->] (B) to node[right]{$J_{n-1}$} (D);
\draw[->] (C) to node[below]{$a_n$} (D);
\end{tikzpicture}
\end{equation}
so that
\[
C^{(n)}
=
C^{(n-1)}
\amalg_{F_{n-1}(S^{n-1})}
F_{n-1}(D^n).
\]

For every $k\geq n$, let
$\overline{a}^{\sigma,k}_n$ and $\overline{a}^{\tau,k}_n$ be the
composites
\[
D^n
\xrightarrow{(a_n)_{n-1}}
C^{(n)}_{n-1}
\xrightarrow{\sigma^{k-n+1}_*}
C^{(n)}_k
\]
and
\[
D^n
\xrightarrow{(a_n)_{n-1}}
C^{(n)}_{n-1}
\xrightarrow{\tau^{k-n+1}_*}
C^{(n)}_k,
\]
respectively.

We now define
$\lambda^{(n)}_k,\rho^{(n)}_k:D^k\rightrightarrows C^{(n)}_k$ for every
$k\geq n$. Using the maps constructed at the previous stage, define
\[
\lambda^{(n)}_k
=
\left(
(J_{n-1})_k\circ\lambda^{(n-1)}_k,
\overline{a}^{\tau,k}_n
\right)
\circ
c^l_{n,k-n}
\]
and
\[
\rho^{(n)}_k
=
\left(
\overline{a}^{\sigma,k}_n,
(J_{n-1})_k\circ\rho^{(n-1)}_k
\right)
\circ
c^r_{n,k-n}.
\]
Explicitly,
\[
D^k
\xrightarrow{c^l_{n,k-n}}
D^k\amalg_{D^{n-1}}D^n
\xrightarrow{
\left(
(J_{n-1})_k\circ\lambda^{(n-1)}_k,
\overline{a}^{\tau,k}_n
\right)}
C^{(n)}_k
\]
defines $\lambda^{(n)}_k$, while
\[
D^k
\xrightarrow{c^r_{n,k-n}}
D^n\amalg_{D^{n-1}}D^k
\xrightarrow{
\left(
\overline{a}^{\sigma,k}_n,
(J_{n-1})_k\circ\rho^{(n-1)}_k
\right)}
C^{(n)}_k
\]
defines $\rho^{(n)}_k$.

It remains to verify that $\lambda^{(n)}_n$ and $\rho^{(n)}_n$
are parallel. Pre-composing with the source map
$\sigma:D^{n-1}\to D^n$ and using the defining source compatibility
of the algebraic left and right whiskering operations gives
\[
\lambda^{(n)}_n\circ\sigma
=
(J_{n-1})_n\circ\lambda^{(n-1)}_n\circ\sigma
\]
and
\[
\rho^{(n)}_n\circ\sigma
=
\overline{a}^{\sigma,n}_n\circ\sigma.
\]
The pushout defining $C^{(n)}$ identifies the source of the newly
attached $n$-cell with the source of $\lambda^{(n-1)}_n$ which is identified with $\rho^{(n-1)}_{n-1}$. Hence
\[
\overline{a}^{\sigma,n}_n\circ\sigma
=
(J_{n-1})_n\circ\lambda^{(n-1)}_n\circ\sigma.
\]
Therefore
\[
\lambda^{(n)}_n\circ\sigma
=
\rho^{(n)}_n\circ\sigma.
\]

Similarly, pre-composing with the target map
$\tau:D^{n-1}\to D^n$ gives
\[
\lambda^{(n)}_n\circ\tau
=
\overline{a}^{\tau,n}_n\circ\tau
\]
and
\[
\rho^{(n)}_n\circ\tau
=
(J_{n-1})_n\circ\rho^{(n-1)}_n\circ\tau.
\]
The pushout identifies the target of the newly attached $n$-cell with the target of
$\rho^{(n-1)}_n$ which is identified with $\rho^{(n-1)}_{n-1}$, so
\[
\overline{a}^{\tau,n}_n\circ\tau
=
(J_{n-1})_n\circ\rho^{(n-1)}_n\circ\tau.
\]
Thus
\[
\lambda^{(n)}_n\circ\tau
=
\rho^{(n)}_n\circ\tau.
\]

Therefore $\lambda^{(n)}_n$ and $\rho^{(n)}_n$ are parallel, and therefore
determines a map
\[
\beta_n:S^n\to C^{(n)}_n.
\]
By Lemma \ref{Yoneda_properties_2}, this induces
\[
b_n:F_n(S^n)\to C^{(n)}.
\]
Finally, we have that the following calculation holds.
\begin{equation}\label{formula_needed}
C^{(n)}_t=
\begin{cases}
C^{(t+1)}_t, & t<n-1,\\[2pt]
C^{(n-1)}_{n-1}\amalg_{S^{n-1}}D^{n}, & t=n-1,\\[2pt]
C^{(n-1)}_t\amalg_{S^{n-1}\amalg S^{n-1}}
(D^{n}\amalg D^{n}), & t\geq n
\end{cases}
\end{equation}

This completes the inductive step. By induction, we obtain a sequence
\[
C^{(0)}
\to
C^{(1)}
\to
C^{(2)}
\to
C^{(3)}
\to\cdots
\]
of co-globular infinity groupoids whose colimit we call the \emph{canonical cylinder object} and denote by $\mathbf{Cyl}_\bullet$. By induction and construction, we have that
\[
\mathbf{Cyl}_n=C^{(n+1)}_n
\]
for all $n\geq 0$. Since $C^{(0)}=D_\bullet\coprod D_\bullet$, we obtain an induced map from the colimit.
\[
\mathbf{coev}:D_\bullet\coprod D_\bullet\to\mathbf{Cyl}_\bullet
\]
The coproduct comes equipped with built-in structure maps 
\[
\psi_0,\psi_1:D_\bullet\rightrightarrows D_\bullet\coprod D_\bullet
\]
which pre-composes with $\mathbf{coev}$ to obtain the following structure maps.
\[
\mathbf{coev}_0,\mathbf{coev}_1:D_\bullet\rightrightarrows\mathbf{Cyl}_\bullet
\]

\subsection*{Unraveling the Cylinder Object}
We now need to unravel the data of the cylinder object we built and maps out of them. To do so, we give the following definitions to make the work easier to follow. We implement notation that copies the notation of Lanari in \cite{Lanari2020}.
\begin{definition}
The \emph{$n$-cylinder representative} is the infinity groupoid 
\[
\mathbf{Cyl}_n=C^{(n+1)}_n
\]
for all $n\geq 0$. The \emph{boundary of the $n$-cylinder representative}, denoted by $\partial(\mathbf{Cyl}_n)$, is defined to be $C^{(n)}_{n}$ for all $n\geq 0$.
\end{definition}

\begin{definition}
Given an infinity groupoid $X$, an $n$-cylinder in $X$ is a map 
\[
C:\mathbf{Cyl}_n\to X
\]
for all $n\geq 0$. The boundary of $C$, denoted by $\partial C$, is given by the inclusion 
\[
f\circ (J_n)_n.
\]
Given an $n$-cylinder $C$ of $X$, we write the following notation.
\[
C\circ \psi_0=C\circ \mathbf{coev}_0\qquad C\circ \psi_1=C\circ \mathbf{coev}_1\qquad \mathbf{Cyl}^s_n=C^{n-1}_n
\]
\[
s(C)=C\circ \mathbf{Cyl}(s)\qquad t(C)=C\circ \mathbf{Cyl}(t)
\]
Given an $n$-cylinder $C$ with $s(C)=A$ and $t(C)=B$, we will write $C:A\rightsquigarrow B$ to denote the cylinder.
\end{definition}

\begin{proposition}
The canonical map
\[
\iota:
D_\bullet\amalg D_\bullet\to\mathbf{Cyl}_\bullet
\]
is a point-wise cofibration.
\end{proposition}

\begin{proof}
We prove that the canonical map
\[
D_\bullet\amalg D_\bullet\to C^{(n)}
\]
is a point-wise cofibration by induction on $n$. For $n=0$, this map is the identity. Now suppose that
\[
D_\bullet\amalg D_\bullet\to C^{(n)}
\]
is a point-wise cofibration. By construction, we have a pushout
\[
\begin{tikzpicture}
\node (A) at (0,2) {$F_n(S^n)$};
\node (B) at (4,2) {$C^{(n)}$};
\node (C) at (0,0) {$F_n(D^{n+1})$};
\node (D) at (4,0) {$C^{(n+1)}$};

\draw[->] (A) -- (B);
\draw[->] (A) -- (C);
\draw[->] (B) -- (D);
\draw[->] (C) -- (D);
\end{tikzpicture}
\]
and therefore 

\[
C^{(n+1)}_k=
\begin{cases}
C^{(k+1)}_k, & k<n,\\[2pt]
C^{(n)}_{n}\amalg_{S^n}D^{n+1}, & k=n,\\[2pt]
C^{(n)}_k\amalg_{S^n\amalg S^n}
(D^{n+1}\amalg D^{n+1}), & k\geq n+1,
\end{cases}
\]
so that 
\[
D_\bullet\amalg D_\bullet\to C^{(n+1)}
\]
is a point-wise cofibration. By induction and upon taking colimits, we obtain that 
\[
\iota:
D_\bullet\amalg D_\bullet\to\mathbf{Cyl}_\bullet
\]
is a point-wise cofibration.
\end{proof}

The following theorem now holds by the formulas we have obtained in this paper and certain pushouts that Lanari obtains in \cite{Lanari2020}.

\begin{theorem}\label{iso_to_Lan}
There is an isomorphism of of co-globular infinity groupoids $\phi:\mathbf{Cyl}_\bullet\to \mathbf{E}$, where $\mathbf{E}$ is the cylinder construction given by Lanari.
\end{theorem}

\begin{proof}
The proof is completed with induction by utilizing Calculation \ref{formula_needed} together with Diagram (9) and Diagram (10) of \cite{Lanari2020}. 
\end{proof}

Therefore, we have provided a different construction for the cylinder object. Moreover, the construction provided here is isomorphic to the one Lanari gave.

\begin{references*}
\bibitem[Ad\'{a}mek and Rosick\'{y}, 1994]{AdRo}
J.~Ad\'{a}mek and J.~Rosick\'{y},
\emph{Locally Presentable and Accessible Categories}.
Cambridge University Press, Cambridge, 1994.

\bibitem[Ara, 2012]{Ara2013Theta}
D.~Ara,
\emph{The groupoidal analogue $\Theta_e$ to Joyal's category $\Theta$
is a test category}.
Applied Categorical Structures,
20(6):603--649, 2012.
doi: \url{https://doi.org/10.1007/s10485-011-9262-5}.
Available at \url{https://arxiv.org/abs/1012.4319}.

\bibitem[Ara, 2013]{Dim1}
D.~Ara,
\emph{On the homotopy theory of Grothendieck $\infty$-groupoids}.
Journal of Pure and Applied Algebra,
217(7):1237--1278, 2013.
doi: \url{https://doi.org/10.1016/j.jpaa.2012.10.016}.
Available at \url{https://www.normalesup.org/~ara/files/homgr.pdf}.

\bibitem[Bourke and Garner, 2013]{BoGa}
J.~Bourke and R.~Garner,
\emph{On semiflexible, flexible and pie algebras},
Journal of Pure and Applied Algebra,
217(2) (2013), 293--321.

\bibitem[Bourke, 2016]{Bo2}
J.~Bourke,
\emph{Note on the construction of globular weak omega-groupoids from
types, topological spaces, $\ldots$}.
Cahiers Topologie G\'eom. Diff\'erentielle Cat\'eg.
\textbf{57} (2016), no.~4, 281--294.

\bibitem[Bourke and Garner, 2016]{BourkeGarner2016}
J.~Bourke and R.~Garner,
\emph{Algebraic weak factorisation systems I: accessible AWFS}.
Journal of Pure and Applied Algebra
\textbf{220} (2016), 108--147.
Available at \url{https://arxiv.org/abs/1412.6559}.

\bibitem[Garner, 2009]{Garner2009}
R.~Garner,
\emph{Understanding the Small Object Argument}.
Appl. Categ. Structures
\textbf{17} (2009), no.~3, 247--285.

\bibitem[Grothendieck, 2022]{Gr1}
A.~Grothendieck,
\emph{Pursuing Stacks (\`a la poursuite des champs). Vol.~I}.
Documents Math\'ematiques~20.
Soci\'et\'e Math\'ematique de France, Paris, 2022.

\bibitem[Henry, 2016]{Henry2016}
S.~Henry,
\emph{Algebraic Models of Homotopy Types and the Homotopy Hypothesis}.
Available at arXiv:1609.04622 [math.AT], 2016.

\bibitem[Henry and Lanari, 2023]{HenLan}
S.~Henry and E.~Lanari,
\emph{On the homotopy hypothesis for 3-groupoids}.
Theory Appl. Categ.
\textbf{39} (2023), Paper No.~26, 735--768.

\bibitem[Lanari, 2018]{Lanari2018}
E.~Lanari,
\emph{A semi-model structure for Grothendieck weak 3-groupoids}.
arXiv:1809.07923, 2018.
Available at \url{https://arxiv.org/abs/1809.07923}.

\bibitem[Lanari, 2020]{Lanari2020}
E.~Lanari,
\emph{Towards a globular path object for weak $\infty$-groupoids}.
Journal of Pure and Applied Algebra,
224(2):630--702, 2020.
doi: \url{https://doi.org/10.1016/j.jpaa.2019.06.004}.
Available at \url{https://www.sciencedirect.com/science/article/pii/S0022404919301501}.

\bibitem[Maltsiniotis, 2005]{Geo2}
G.~Maltsiniotis,
\emph{La th\'eorie de l'homotopie de Grothendieck}.
Ast\'erisque \textbf{301} (2005), vi+140 pp.

\bibitem[Maltsiniotis, 2010]{Malt}
G.~Maltsiniotis,
\emph{Grothendieck $\infty$-groupoids, and still another definition of
$\infty$-categories}.
Available at arXiv:1009.2331 [math.CT], 2010.

\bibitem[Quillen, 1967]{Quillen1967}
D.~G. Quillen,
\emph{Homotopical Algebra}.
Lecture Notes in Mathematics, Vol.~43,
Springer-Verlag, 1967.

\bibitem[Rezk, 2001]{Rezk2001}
C.~Rezk,
\emph{A model for the homotopy theory of homotopy theory}.
Trans. Amer. Math. Soc. \textbf{353} (2001), no.~3, 973--1007.
Available at \url{https://arxiv.org/abs/math/9811037}.

\end{references*}
\end{document}